\documentclass{amsart}
\usepackage{amsthm,amsmath,amsfonts,amssymb,amscd,mathrsfs,graphics}

\usepackage{enumerate}
\usepackage{listings}

\usepackage{todonotes}

\newtheorem{thm}{Theorem}[section]
\newtheorem{cor}{Corollary}[section]
\newtheorem{lemma}{Lemma}[section]

\makeatletter
\renewcommand{\@seccntformat}[1]{\S{\csname
the#1\endcsname}\hspace{0.5em}}
\makeatother

\begin{document}

\title{Strong Gelfand Pairs of the Sporadic Groups, and their Extensions}

\author{Joseph E. Marrow}
  \address{Department of Mathematics,  Southern Utah University, Cedar City, 
UT 84720, U.S.A.
E-mail:  josephmarrow@suu.edu}
\date{}
\maketitle

\begin{abstract}  
A strong Gelfand pair $(G, H)$ is a finite group $G$ and a subgroup $H$ where every irreducible character of $H$ induces to a multiplicity-free character of $G$. We determine the strong Gelfand pairs of the sporadic groups, their automorphism groups, and their covering groups. We also find the (strong) Gelfand pairs of the generalized Mathieu groups, the Tits group, and the automorphism group of the Leech Lattice.

\medskip

\medskip
\end{abstract}

\theoremstyle{plain}
\numberwithin{equation}{section}

\setlength{\leftmargini}{1.em} \setlength{\leftmarginii}{1.em}
\renewcommand{\labelenumi}{\setlength{\labelwidth}{\leftmargin}
   \addtolength{\labelwidth}{-\labelsep}
   \hbox to \labelwidth{\theenumi.\hfill}}

\maketitle

\section{Introduction}

In \cite{lux1} the multiplicity-free permutation characters of the sporadic groups are found. This work was continued in \cite{lux2} where the multiplicity-free permutation characters of the extensions of the sporadic groups are given. 

A \textit{Gelfand pair} $(G, H)$ is a pair of groups $H\leq G$ where $\langle \chi\downarrow H, 1_H\rangle \leq 1$ for all irreducible characters $\chi$ of $G$. This is equivalent to the permutation representation of $G$ on the cosets of $H$ is multiplicity-free. Which means that \cite{lux1} gives the Gelfand pairs of the sporadic groups.  There are other equivalent conditions which we will not list here, but may be found in in \cite{hu}.

A strong Gelfand pair $(G, H)$ is a pair of groups $H\leq G$ where $\langle \chi\downarrow H, \psi\rangle\leq 1$ for all irreducible characters $\chi$ of $G$ and $\psi$ of $H$. In \cite{thesis}, an attempt was made to find all of the strong Gelfand pairs for the sporadic groups. This was incomplete, as the author was unable to determine the behavior of the sporadic groups of order larger than $10^{14}$. 

In this paper, we will finished categorizing the strong Gelfand pairs of the sporadic groups. We will also consider the (strong) Gelfand pairs of groups closely related to the sporadic groups. In \cite{grady} the idea of an extra-strong Gelfand pair is introduced. We do not consider these. 

It is always that case that $(G, G)$ is a strong Gelfand pair. In what follows, we only consider those pairs $(G, H)$ where $H\lneq G$. We define, for convenience, a \textit{(strong) Gelfand subgroup} to be a proper subgroup $H\lneq G$ where $(G, H)$ is a (strong) Gelfand pair.

By Frobenius Reciprocity \cite{jl} we obtain Lemma \ref{stack}
\begin{lemma}\cite[Lemma 3.1]{hu}\label{stack}
Suppose $K \leq H \leq G$ are groups. If $(G, H)$ is not a strong Gelfand pair, then neither is $(G, K)$.
\end{lemma}

Because of this, we will often discuss maximal subgroups, making use of the contrapositive.

An additional result which will be of use to us is Lemma \ref{normal}
\begin{lemma}\cite[Lemma 3.2]{hu}\label{normal}
Let $N \leq H \leq G$ with $N\trianglelefteq G$. Then $(G, H)$ is a (strong) Gelfand pair if and only if $(G/N, H/N)$ is a (strong) Gelfand pair.
\end{lemma}

Throughout, we will freely use so-called ATLAS notation \cite{atlas}.

\section{The Strong Gelfand Pairs of the Sporadic Groups}

\begin{thm}
The strong Gelfand pairs of the sporadic groups are:
\begin{center}
\begin{tabular}{cccc}
\centering
$(M_{11}, M_{10});$ & $(M_{11}, L_2(11));$ & $(M_{12}, M_{11});$ & $(M_{22}, L_3(4)).$
\end{tabular}
\end{center}
\end{thm}

\begin{proof} By the definition it is only necessary to consider the known Gelfand pairs given in \cite{lux1}. Our result follows in the same manner as given in \cite{thesis}, being an exhaustive search, however we used GAP \cite{GAP} for the calculations. An example of our code is provided in \S\ref{apx1}. Both non-conjugate copies of $M_{11}$ make $(M_{12}, M_{11})$ a strong Gelfand pair.
\end{proof}



\section{Automorphism and Covering Groups}
In \cite{lux1} the Gelfand pairs of the automorphism groups of each of the sporadic groups are given. In interest of completeness, we include the strong Gelfand pairs here for those cases where $G\neq Aut(G)$.

\begin{thm}\label{autosgp}
The strong Gelfand pairs of the automorphism groups of the sporadics are given in Table \ref{automorph}.
\end{thm}

\begin{table}[h!]
\centering
\caption{Strong Gelfand Pairs of the Sporadic Automorphism Groups}
\label{automorph}
\vspace{5pt}
\begin{tabular}{ccc}
Group & Automorphism Group & Strong Gelfand Pairs\\
\hline
$M_{12}$ & $M_{12}.2$ & $(M_{12}.2, M_{12})$\\
 & & $(M_{12}.2, M_{11})$\\
$M_{22}$ & $M_{22}.2$ & $(M_{22}.2, M_{22})$\\
$J_2$ & $J_2.2$ & $(J_2.2, J_2)$\\
$HS$ & $HS.2$ & $(HS.2, HS)$\\
$J_3$ & $J_3.2$ & $(J_3.2, J_3)$\\
$McL$ & $McL.2$ & $(McL.2, McL)$\\
$He$ & $He.2$ & $(He.2, He)$\\
$Suz$ & $Suz.2$ & $(Suz.2, Suz)$\\
$ON$ & $ON.2$ & $(ON.2, ON)$\\
$Fi_{22}$ & $Fi_{22}.2$ & $(Fi_{22}.2, Fi_{22})$\\
$HN$ & $HN.2$ & $(HN.2, HN)$\\
$Fi_{24}^\prime$ & $Fi_{24}^\prime.2$ & $(Fi_{24}^\prime.2,Fi_{24}^\prime)$
\end{tabular}
\end{table}

Both copies of $M_{11}$ for which $(M_{12}, M_{11})$ is a strong Gelfand pair make $(M_{12}.2, M_{11})$ a strong Gelfand pair, this because the two distinct $M_{11}$ are interchanged by an outer automorphism, and thus are conjugate in $M_{12}.2$.

Theorem \ref{autosgp} leads us to consider a more interesting result, which\textemdash while immediate\textemdash the author doesn't believe it has appeared in print previously.

\begin{thm}\label{ext}
For $p$ a prime, and $G$ a finite group, $(G.p, G)$ is a strong Gelfand pair.
\end{thm}
\begin{proof}
Using Lemma \ref{normal}, since $G \leq G \leq G.p$ and $G\trianglelefteq G.p$ we see $(G.p/G, G/G) = (p, 1)$ is a strong Gelfand pair. Then so is $(G.p, G)$.
\end{proof}

Theorem \ref{ext} can be generalized.


\begin{cor}\label{abelext}
For $A$ an abelian group, and $G$ a finite group, $(G.A, G)$ is a strong Gelfand pair.
\end{cor}

We also consider the extensions by sporadic groups. This fully extends the classification of Gelfand pairs given in \cite{lux1, lux2} to strong Gelfand pairs. 

\begin{thm}\label{dcover}
The double cover $2.M_{12}$ has four strong Gelfand pairs: $(2.M_{12}, 2\times M_{11})$, and $(2.M_{12}, M_{11})$, twice each.
\end{thm}
\begin{proof}
A separate file containing the code for these\textemdash and the non-existance of others\textemdash is provided.
\end{proof}
No other sporadic extension has any strong Gelfand pairs, and the Gelfand pairs can be found in \cite{lux2}.

\section{Mathieu Groups}
There are Mathieu groups which are not simple. The sporadic Mathieu groups are transitive permutation groups, but the generalized Mathieu groups are stabilizers within the sporadic Mathieu groups. We saw one previously in the strong Gelfand pair $(M_{11}, M_{10})$. For the remaining Mathieu groups, we provide all Gelfand and strong Gelfand subgroups.

\begin{thm}
All Gelfand subgroups of $M_7$, $M_8$, $M_9$, $M_{10}$, $M_{20}$, and $M_{21}$ are given in Table \ref{mathgptable}.
\end{thm}

\begin{table}[h!]
\centering
\caption{Gelfand Subgroups of the Generalized Mathieu Groups}
\label{mathgptable}
\vspace{5pt}
\begin{tabular}{cccccc}
Group & Order & Iso. type & Gel. Subgroups & Max. Subgroup & Multiplicity\\
\hline
$M_7$ & $1$ & $1$ & $1$ & Improper & Improper\\
$M_8$ & $8$ & $Q_8$ & $4$ & - & $3$\\
 &  &  & $2$ & $4$ & $1$\\
$M_9$ & $72$ & $\mathrm{PSU}_3(2)$ & $3^2\colon 4$ & - & $3$\\
$M_{10}$ & $720$ & $A_6.2$ & $A_6$ & - & $1$\\
$M_{20}$ & $960$ & $2^4\colon A_5$ & $2^4\colon \mathrm{D}_{10}$ & - & $1$\\
 & & & $2^4\colon 5$ & $2^4\colon \mathrm{D}_{10}$ & $1$\\
 & & & $4^2\colon A_4$ & - & $1$\\
 & & & $4^2\colon 3$ & $4^2\colon A_4$ & $2$\\
 & & & $2^4\colon 3$ & $4^2\colon A_4$ & $2$\\
$M_{21}$ & $20160$ & $\mathrm{L}_3(4)$ & $M_{20}$ & - & $2$
\end{tabular}
\end{table}

\begin{thm}
The strong Gelfand pairs of the generalized Mathieu groups $M_7$, $M_9$, $M_{10}$, $M_{20}$, and $M_{21}$ are precisely the Gelfand pairs of the generalized Mathieu groups, and thus are had in Table \ref{mathgptable}. For $M_8$ only the maximal subgroups are strong Gelfand subgroups.
\end{thm}
\begin{proof}
Since $M_7$ is the trivial group, it only has the trivial strong Gelfand pair. $M_8$ is isomorphic to the quaternion group, and thus is covered in \cite{thesis, dihedral}. The remainder are obtained by calculations \cite{GAP}; it is useful to know that $M_{20}$ is the largest maximal subgroup of $M_{21}$.
\end{proof}

In Table \ref{mathgptable} we provide an additional name for the generalized Mathieu groups in the `Isomorphism Type' column. The `Multiplicity' column of Table \ref{mathgptable} counts the number of times this subgroup gives a Gelfand pair. For example, the group $M_8=\langle a, b\,\vert\, a^2=b^2=(ab)^2, a^4=1, ab=-ba\rangle$ has four Gelfand pairs $(M_8, \langle a\rangle),$ $(M_8, \langle b\rangle),$ $(M_8, \langle ab\rangle)$, and $(M_8, \langle a^2\rangle)$. The first three of these are maximal. If a subgroup is not maximal we list the maximal subgroup to which it belongs in the appropriate column.

The (strong) Gelfand pairs of the automorphism groups of the generalized Mathieu groups have been considered, but owing to the length we have elected to include them in the file containing the code for Theorem \ref{dcover}.

\section{Automorphism Group of the Leech Lattice}
The automorphism group of the Leech lattice, sometimes denoted $\bullet 0$ or $Co_0$, is the group $2.Co_1$ \cite{conwaylattice}. It has order $8315553613086720000$ and was first described in \cite{conwaydot}.

\begin{thm}
The only subgroups of $Co_0$ which are Gelfand subgroups are given in Table \ref{leechgptable}.
\end{thm}

\begin{table}[h!]
\centering
\caption{Gelfand Subgroups of $Co_0$}
\label{leechgptable}
\vspace{5pt}
\begin{tabular}{ccccc}
Subgroup & Order & Maximal Subgroup\\
\hline
$2.Co_2$ & $84610842624000$ & - \\
$Co_2$ & $42305421312000$ & $2.Co_2$ \\
$6.Suz.2$ & $5380145971200$ & - \\
$6.Suz$ & $2690072985600$ & $6.Suz.2$\\
$2.2^{11}\colon M_{24}$ & $1002795171840$ & - \\
$2.Co_3$ & $991533312000$ & - \\
$Co_3$ & $495766656000$ & $2.Co_3$ \\
$2.2^{(1+8)}_+.\mathrm{O}^+_8(2)$ & $1321205760$ & - \\
\end{tabular}
\end{table}
\begin{proof} 
All maximal subgroups of $2.Co_1$ contain the center $Z(2.Co_1) = 2$. Then Lemma \ref{normal} shows that a maximal subgroup $2.H\leq 2.Co_1$ forms the Gelfand pair $(2.Co_1, 2.H)$ if and only if $(Co_1, H)$ is a Gelfand pair. Referencing \cite{lux1}, this means that $2.Co_2$, $6.Suz.2$, $2.2^{11}\colon M_{24}$, $2.Co_3$, and $2.2^{1+8}_+.\mathrm{O}^+_8(2)$ are Gelfand subgroups.

As the subgroups being maximal was not needed to determine which are Gelfand subgroups, we also get $(2.Co_1, 6.Suz)$ is a Gelfand pair.

In \cite{lux2} the Gelfand pairs are listed as $(2.Co_1, Co_2)$ and $(2.Co_1, Co_3)$, these being the only two which cannot be gotten by the quotient argument.
\end{proof}


It is worth being aware that the group $2.2^{11}\colon M_{24}$ is the group $N$ refered to in \cite[Chapter 10]{conwaylattice}.

\begin{thm}\label{noco}
There is no proper subgroup $H\leq Co_0$ for which $(Co_0, H)$ is a strong Gelfand pair.
\end{thm}

\section{The Tits Group, and its Automorphism Group}
The Tits group $^2F_4(2)^\prime$, named for Jacques Tits, is a simple group of order $17971200$. It is sometimes categorized as a $27$th sporadic group, due to its not quite fitting into the category of groups of Lie type (for one, it has no $BN$-pair \cite{tits}). For this reason we have included it here. The maximal subgroups of the Tits group are known from \cite{atlas, titsmax2, titsmax1}. 

\begin{thm}\label{titsgp}
The Gelfand subgroups of the Tits group are given in Table \ref{titsgptable}.
\end{thm}

Where the subgroup is itself maximal, we leave the maximal subgroup column empty.

\begin{table}[h!]
\centering
\caption{Gelfand Subgroups of the Tits Group}
\label{titsgptable}
\vspace{5pt}
\begin{tabular}{ccc}
Subgroup & Order & Maximal Subgroup\\
\hline
$\mathrm{L}_2(25)$ & $7800$ & -\\
$\mathrm{L}_3(3)$ & $5616$ & $\mathrm{L}_3(3)\colon 2$\\
$\mathrm{L}_3(3)$ & $5616$ & $\mathrm{L}_3(3)\colon 2$\\
$\mathrm{L}_3(3)\colon 2$ & $11232$ & -\\
$\mathrm{L}_3(3)\colon 2$ & $11232$ & -\\
$2^5.2^3.A_4$ & $3072$ & $2^5.2^3.S_4$\\
$2^5.2^3.S_4$ & $6144$ & -\\
$2^5.2^4.F_5$ & $10240$ & -
\end{tabular}
\end{table}

\begin{thm}\label{titsnosgp}
The Tits group has no proper subgroups $H$ for which $\left(^2 F_4(2)^\prime, H\right)$ is a strong Gelfand pair.
\end{thm}

Where we've already considered the automorphism groups of the sporadic groups, it seems appropriate to finish by considering the (strong) Gelfand pairs of the automorphism group of the Tits group. Since $Aut\left(^2F_4(2)^\prime\right) =\,^2F_4(2)$ we also consider this group, the largest proper subgroup of $Ru$.

\begin{thm}\label{resssgp}
The only strong Gelfand subgroup of $^2F_4(2)$ is $^2F_4(2)^\prime$.
\end{thm}
\begin{proof} That this is a strong Gelfand subgroup immediately follows from Theorem \ref{ext}. The lack of any others comes from GAP \cite{GAP}. \end{proof}

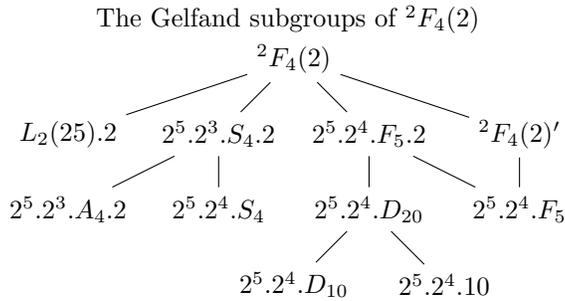
\begin{table}[h!]
\begin{center}
The Gelfand subgroups of $^2F_4(2)$
\begin{tikzpicture}
\node (1) at (0,3) {$^2 F_4(2)$}; %
\node (2) at (-1,2) {$2^5.2^3.S_4.2$}; %
\node (3) at (-3,2) {$L_2(25).2$}; %
\node (4) at (1,2) {$2^5.2^4.F_5.2$}; 
\node (5) at (3,2) {$^2 F_4(2)^\prime$}; %
\node (6) at (-3,1) {$2^5.2^3.A_4.2$}; 
\node (7) at (-1,1) {$2^5.2^4.S_4$}; 
\node (8) at (1,1) {$2^5.2^4.D_{20}$}; 
\node (9) at (3,1) {$2^5.2^4.F_5$}; 
\node (10) at (0,0) {$2^5.2^4.D_{10}$}; 
\node (11) at (2,0) {$2^5.2^4.10$}; 
\draw[-] (3) -- (1) -- (5) -- (9) -- (4) --(8) -- (11);
\draw[-] (8) -- (10);
\draw[-] (6)--(2)--(1)--(4);
\draw[-] (2) -- (7);
\end{tikzpicture}
\end{center}
\end{table}

\begin{thm}
The group $^2F_4(2)$ has $10$ Gelfand pairs.
\end{thm}
\begin{proof}
These are quickly found using MAGMA \cite{magma}. We wish to note that for $G=\,^2F_4(2)\geq \,^2F_4(2)^\prime \geq 2^5.2^4.F_5= H$ the pair $(G, H)$ is a Gelfand pair, and that $H$ is also a subgroup of the maximal subgroup $2^5.2^4.F_5.2$
 (where $F_5$ is the Frobenius group of order $20$), showing that $(G, H.2)$ is a Gelfand pair. This causes the subgroup relations of the Gelfand subgroups to not form a tree, as shown in the figure.
\end{proof}

\section{Appendix}\label{apx1}
We provide here GAP code to show which maximal subgroups of the sporadics are strong Gelfand pairs. The Baby Monster and Monster groups are not included here, as they are best shown by specifically targeting the potential subgroups given in \cite{lux1} rather than searching every maximal subgroup. Similarly, the code provided here does not show that subgroups of maximal subgroups are not strong Gelfand. Since, for sporadic groups, the only cases where that could happen are already known by \cite{thesis}, we are done. 

Modifications of this can be used to verify all other results. 
\begin{lstlisting}
for string in ["M11", "M12", "M22", "M23", "M24", "J1",
"J2", "J3", "J4", "Co1", "Co2", "Co3", "Fi22", "Fi23",
"Fi24'", "HS", "McL", "He", "Ru", "Suz", "ON", "HN",
"Ly", "Th"] do
 ctg:=CharacterTable(string);
 for h in Maxes(ctg) do
  cth:=CharacterTable(h);
  chars:=[];
  tf:=true;
  for i in [1..Length(Irr(cth))] do
   Append(chars, [Irr(cth)[i]^ctg]);
  od;
  for i in [1..Length(Irr(cth))] do
   for j in [1..Length(Irr(ctg))] do
    if ScalarProduct(ctg, chars[i], Irr(ctg)[j]) > 1 then
     tf:=false;
    fi;
   od;
  od;
  Print("(", string, ",", h, ") ", tf, " \n");
 od;
od;
\end{lstlisting}

We found MAGMA was easier to work with when dealing with the Tits group. We provide below code which will list every Gelfand subgroup of $^2F_4(2)$.

\begin{lstlisting}
G:=TitsGroup();
A:=PermutationGroup(AutomorphismGroup(G));

gp:=function(g, h);
	tf:=true;
	ctg:=CharacterTable(g);
	cth:=CharacterTable(h);
	ind:=Induction(cth[1], A);
	for i:=1 to #ctg do
		if InnerProduct(ind, ctg[i]) gt 1 then
			tf:=false;
			break i;
		end if;
	end for;
	return tf;
end function;

ms:=[u`subgroup : u in Subgroups(A)];
for sub in ms do
	if gp(sub) then
		GroupName(sub);
	end if;
end for;
\end{lstlisting}

\end{document}